\documentclass[10pt]{article}
\usepackage{geometry}               % See geometry.pdf to learn the layout options. There are lots.
\usepackage[english]{babel}
\usepackage[latin1]{inputenc}
\usepackage{amssymb,amsmath,amsthm,amsfonts}
\usepackage{graphicx}
\usepackage[usenames,dvipsnames]{xcolor}
\usepackage{cancel}
\usepackage[colorlinks]{hyperref}
\newcommand{\N}{\mathbb{N}}

\newcommand{\R}{\mathbb{R}}

\newcommand{\eps}{\varepsilon}

\newcommand{\Bcal}{{\mathcal{B}}}

\newcommand{\Ucal}{{\mathcal{U}}}

\def\XXint#1#2#3{{\setbox0=\hbox{$#1{#2#3}{\int}$ }
\vcenter{\hbox{$#2#3$ }}\kern-.6\wd0}}

\DeclareMathOperator{\dist}{dist}

\newtheorem{proposition}{Proposition}[section]
\newtheorem{theorem}[proposition]{Theorem}
\newtheorem{corollary}[proposition]{Corollary}
\newtheorem{lemma}[proposition]{Lemma}

\theoremstyle{definition}

\newtheorem{remark}[proposition]{Remark}

\numberwithin{equation}{section}
\newcommand{\beq}{\begin{equation}}
\newcommand{\eeq}{\end{equation}}
\newcommand{\ben}{\begin{enumerate}}
\newcommand{\een}{\end{enumerate}}
\newcommand{\bit}{\begin{itemize}}
\newcommand{\eit}{\end{itemize}}

\title{Normalized solutions for Sobolev critical Schr\"{o}dinger equations on bounded domains}

\author{Dario Pierotti, Gianmaria Verzini and Junwei Yu}

\begin{document}
\maketitle

\begin{abstract}
We study the existence and multiplicity of positive solutions with prescribed $L^2$-norm for 
the Sobolev critical Schr\"odinger equation on a bounded domain $\Omega\subset\mathbb{R}^N$, $N\ge3$: 
\[
-\Delta U = \lambda U + U^{2^{*}-1},\qquad U\in H^1_0(\Omega),\qquad
\int_\Omega U^2\,dx = \rho^{2},
\]
where $2^*=\frac{2N}{N-2}$.

First, we consider a general bounded domain $\Omega$ in dimension $N\ge3$, with a restriction, only 
in dimension $N=3$, involving its inradius and first Dirichlet eigenvalue. In this general case 
we show the existence of a mountain pass solution on the $L^2$-sphere, for $\rho$ belonging to a 
subset of positive measure of the interval $(0,\rho^{**})$, for a suitable threshold $\rho^{**}>0$. 
Next, assuming that $\Omega$ is star-shaped, we extend the previous result to all values  
$\rho\in(0,\rho^{**})$. 

With respect to that of local minimizers, already known in the literature, the existence of mountain 
pass solutions in the Sobolev critical case is much more elusive. In particular, our proofs are based 
on the sharp analysis of the bounded Palais-Smale sequences, provided by a nonstandard adaptation of the Struwe monotonicity trick, that we develop.
\end{abstract}
\noindent
{\footnotesize \textbf{AMS-Subject Classification}}.
{\footnotesize 35J20; 35B33; 35Q55; 35J61.}\\
{\footnotesize \textbf{Keywords}}.
{\footnotesize Nonlinear Schr\"odinger equations, constrained critical points, critical exponent, solitary waves, Struwe monotonicity trick.}

\section{Introduction}

Let $\Omega\subset\R^N$, $N\ge3$ be a bounded domain with smooth boundary. 
In this paper we deal with the existence and multiplicity of normalized 
solutions to the Sobolev critical (stationary) Schr\"odinger equation in 
$\Omega$, with homogeneous Dirichlet boundary conditions. Namely, given $\rho>0$, 
we look for solutions $(U,\lambda)\in H^1_0(\Omega)\times \R$ of the semilinear 
elliptic problem, of Brezis-Nirenberg type:
\begin{equation}\label{eq:main}
\begin{cases}
-\Delta U = \lambda U + |U|^{2^{*}-2}U  & \text{in }\Omega,\smallskip\\
\int_\Omega U^2\,dx = \rho^{2}, \quad U=0  & \text{on }\partial\Omega.
\end{cases}
\end{equation}
Here, as usual, $2^{*}:=\frac{2N}{N-2}$ denotes the Sobolev critical exponent. 

The motivation to study \eqref{eq:main} comes from its well known relation with 
the existence of standing waves for the evolutive Schr\"odinger equation on a 
bounded domain with pure Sobolev critical power nonlinearity, which in turn 
is of interest because of its applications to nonlinear optics and to the theory 
of Bose-Einstein condensation, also as a limit case of the equation on $\R^N$ 
with confining trapping potential. In particular, here we look for normalized 
solutions, meaning that the chemical potential $\lambda$ is an unknown in the 
equation and the $L^2$-norm of the solution is prescribed a priori. After the 
pioneering contribution by Jeanjean \cite{MR1430506}, in the last decade the study of 
normalized solutions of Schr\"odinger equations has become a hot topic, with a lot of 
contributions in several directions. While below we will describe with more details  
some literature which is related to our results, from this general point of view  
we mention here just a few recent papers, and we refer the interested reader to the 
references therein, for contributions e.g. about combined nonlinearities \cite{MR4107073,MR4096725,MR4433054,MR4476243}; equations with potentials \cite{MR4304693,MR4443784}; global bifurcation and/or topological approach \cite{MR4297197,MR4603876,MR4701352}, also in connection with 
Mean Field Games systems \cite{MR3660463,MR4191345,MR4570263}; equations on metric graphs \cite{MR4404069,MR4132757,MR4601303,MR4371084,MR4241295}; relations between least action and least energy ground states \cite{MR4566699}.

Solutions of \eqref{eq:main} correspond to critical points of the energy functional $E:H^{1}_{0}(\Omega)\rightarrow \mathbb{R}$, defined by
\begin{equation}\label{func:main0}
E(u)=\frac{1}{2}\int_{\Omega}|\nabla u|^{2}dx-\frac{1}{2^{*}}\int_{\Omega}|u|^{2^{*}}dx,
\end{equation}
constrained to the $L^{2}$-sphere
\begin{equation}\label{manifold:0}
M_{\rho}=\left\{U\in H^{1}_{0}(\Omega): \int_{\Omega} U^{2}=\rho^{2}\right\},
\end{equation}
with $\lambda\in\R$ the associated Lagrange multiplier.

The starting point of our work is the fact that, under an explicit smallness 
condition on $\rho>0$,  $E$ has the mountain pass geometry on $M_\rho$. This fact 
is not peculiar only of the Sobolev critical case, but rather it holds true, for a 
bounded domain, in the so-called \emph{mass--supercritical} case, namely when $2^*$ 
is replaced with an exponent $p$ with
\[
2+\frac4N < p \le 2^*
\]
(both in \eqref{eq:main} and in $E$). For this reason, one can hope to obtain two 
distinct nontrivial (actually, strictly positive) solutions of the corresponding problem: a local minimizer, associated to an orbitally stable standing wave for the related evolutive equation, 
and a mountain pass solution (orbitally unstable). To the best of our knowledge, 
both these solutions were first obtained in \cite{ntvAnPDE}, in case $\Omega$ is a 
ball and $p$ is Sobolev subcritical. Subsequently, the existence of a local minimizer 
was extended to the Sobolev subcritical case in generic bounded domains in 
\cite{MR3689156}, and to the Sobolev critical case in bounded domains in 
\cite{MR3918087}:
\begin{theorem}[{\cite[Thm. 1.11]{MR3918087}}]\label{Th:mini}
If $N\geq3$ and $\Omega\subset\R^N$ is a bounded domain, then there exists  
$\rho^{*}>0$ such that if $0<\rho<\rho^{*}$, then \eqref{eq:main} has a positive solution, 
which corresponds to a local minimizer of $E$ on $M_{\rho}$.
\end{theorem}

On the other hand, the existence of a mountain pass solution is  
more elusive. The main difficulty in this direction is to construct a bounded  
Palais-Smale at the mountain pass level. Actually, this was one of the main issues 
faced by Jeanjean in \cite{MR1430506}, where a bounded Palais-Smale sequence at 
the mountain pass level was obtained by exploiting dilations and resorting to the 
so-called  Pohozaev manifold. This strategy, which has been successfully adopted 
in many subsequent problems, all settled in $\R^N$ (see e.g. 
\cite{MR3639521,MR4107073,MR4096725,MR4304693,MR4443784}), does not seem appropriate 
to treat problems on bounded domains. A second, more general strategy to overcome this 
issue is the so-called Struwe monotonicity trick \cite{MR970849,MR926524} 
(see also \cite{JJ_BPS}). Roughly speaking, this argument embeds a given problem into 
a family of problems, depending on a parameter, and it allows to obtain 
bounded Palais-Smale sequences, and hence mountain pass solutions, for almost every value 
of the parameter; then, a fine blow-up analysis may help to get a solution also for the 
original problem. This strategy for normalized solutions on graphs 
has been recently exploited by Chang, Jeanjean and Soave in \cite{JJS_graph_IHP}. 
We have also to mention that, for the Sobolev subcritical case in generic bounded 
domains, the first two authors of the present paper claimed the existence of a mountain pass solution 
in their previous paper \cite[Proposition 4.4]{MR3689156}, but the proof contained there misses  
the boundedness of the Palais-Smale sequence: actually, reasoning as in \cite{JJS_graph_IHP}
and exploiting the blow-up analysis which was one of the main results in \cite{MR3689156}, we 
can fill such gap and conclude the proof, see the end of Section \ref{sec:prel}.

The main results of the present paper concern the existence of a mountain pass solution 
for \eqref{eq:main}. A feature of this problem is that, to apply the monotonicity trick, 
one does not need to introduce an artificial parameter: one can use the mass $\rho$ as a 
parameter, easily obtaining the existence of a bounded Palais-Smale sequence at the mountain pass 
level for almost every mass (in the appropriate range). On the other hand, in the Sobolev critical 
case much effort is needed to investigate the strong convergence of such sequence, as we discuss below. 
An analogous issue has already been faced in the full space $\R^N$, in the case of combined 
nonlinearities \cite{MR4096725,MR4433054,MR4476243}. Given the relations between \eqref{eq:main} 
and the classical 
Brezis-Nirenberg problem \cite{BrezisNirenberg}, as one can expect, the cases $N=3$ and $N\ge4$ show 
some differences; moreover, we obtain a completely satisfactory result in the case of star-shaped 
domains, although, using a refined version of the monotonicity trick, we have also existence results
on generic bounded domains. 

In what follows, $\lambda_1(\omega)$ denotes the first eigenvalue of the Dirichlet 
Laplacian on a bounded domain $\omega\subset\R^N$, and $R_\Omega$ denotes the inradius of $\Omega$, i.e. the  
radius of the largest ball $B_R\subset\Omega$. Our main results are the following.

\begin{theorem}\label{Th:mp2}
Let $\Omega\subset\R^N$, $N\ge3$, be a bounded domain. Moreover, in case $N=3$, let us 
assume that 
\begin{equation}\label{eq:BRBR}
\lambda_{1}(B_{R_\Omega})<\frac{4}{3}\lambda_{1}(\Omega).
\end{equation}
Then there exist $0<\rho^{**}\le \rho^*$ and a measurable set $Q_0\subset(0,\rho^{**})$ such that:
\begin{enumerate}
\item for every $0<\rho\le\rho^{**}$, the set $Q_0\cap \left(0,\rho\right)$ has positive Lebesgue measure
(in particular, $|Q_0|>0$);
\item if $\rho\in \overline{Q}_0$, then \eqref{eq:main} has a second positive solution, distinct 
from the one in Theorem \ref{Th:mini};
\item if either $\rho\in Q_0$, or $\rho\in\overline{Q}_0$ is the limit of an increasing sequence of 
points of $Q_0$, then this second solution is at a mountain pass level for $E$ on $M_{\rho}$.
\end{enumerate}
\end{theorem}
\begin{theorem}\label{Th:mp}
Under the assumptions of Theorem \ref{Th:mp2}, assume furthermore that 
$\Omega$ is star-shaped. Then there exists $0<\rho^{**}\le \rho^*$ such that 
for every $0<\rho<\rho^{**}$, \eqref{eq:main} has a second positive solution, at a mountain pass level 
for $E$ on $M_{\rho}$.
\end{theorem}

Some remarks are in order.

\begin{remark}\label{rmk:BRR}
Assumption \eqref{eq:BRBR} is satisfied, in particular, in case $\Omega\subset B_{R'}$, with
$R'<\frac{2\sqrt{3}}{3}R_\Omega$. In particular, Theorem \ref{Th:mp} holds true in case $\Omega$ 
is a ball.
\end{remark}

\begin{remark}\label{rmk:constants}
Although Theorem \ref{Th:mini} is not new, we will provide a very short alternative proof of it, 
for the sake of completeness and also because all its ingredients are needed also to prove 
the existence of mountain pass solutions. In particular, we will see that in Theorem \ref{Th:mini} 
one can take 
\begin{equation}\label{eq:rho}
\rho^{*}=\left(\frac{2S^{N/2}}{N\lambda_{1}(\Omega)}\right)^{1/2},
\end{equation}
where $S$ denotes the best constant of the Sobolev embedding 
\eqref{ineq:sobolev}. Also $\rho^{**}$ can be estimated explicitly, keeping track of 
the details along the proofs of Lemma \ref{energy:mp} and Section \ref{sec:proof_main_res}.
\end{remark}

\begin{remark}\label{rmk:normvslevinstrongconv}
One of the main difficulties in the proof of both Theorem \ref{Th:mp2} and \ref{Th:mp} is 
to show the strong convergence of the bounded Palais-Smale sequence at the mountain pass level, 
which exists for almost every $\rho\in(0,\rho^{**})$ by the Struwe monotonicity trick. This 
is done by contradiction, using the well-known bubbling description for the Palais-Smale sequences 
of the Brezis-Nirenberg problem, although some adaptations are needed to work with normalized problems.  
In this respect, our two theorems exploit distinct features. 

In Theorem \ref{Th:mp}, the contradiction is obtained by sharp estimates of the energy level along 
the Palais-Smale sequence and at its weak limit, using that every solution $(u,\lambda)$ of \eqref{eq:main} with $u>0$ 
satisfies $\lambda>0$; for this we require $\Omega$ to be star-shaped, and in turn this provides a 
mountain pass solution for almost every $\rho$; finally, solutions for every $\rho\in(0,\rho^{**})$ are 
obtained by approximation, using that the mountain pass level is continuous from the left.

On the other hand, the contradiction in Theorem \ref{Th:mp2} relies on estimates on the $H^1_0$-bound 
of the Palais-Smale sequence, which is based, for a.e. $\rho$, on the (unknown) derivative of the 
mountain pass level with respect to $\rho$. To estimate such derivative we exploit once again sharp 
estimates of the mountain pass level. This is done in our key Lemma \ref{monoty}, which is inspired 
by an original version of the monotonicity trick by Struwe \cite[Lemma 2]{MR1245097}. Since we 
can compare the derivatives of ordered functions only on a set of non-vanishing measure, we obtain 
existence of the mountain pass solution only in the set $Q_0$, described above, and the extension 
to $\overline{Q}_0$ exploits once again the left-continuity of the mountain pass level.
\end{remark}

\begin{remark}\label{rmk:posit}
We stress that in our results we obtain multiplicity of \emph{positive} solutions. For general  
mountain pass constructions, a well-known quantitative argument of Ghoussoub \cite{MR1251958}
allows to localize the Palais-Smale sequences; in particular, if the functional is even, this allows 
the construction of positive mountain pass solutions. On the other hand, the combination of 
this argument with the Struwe monotonicity trick is not completely straightforward, and one 
needs some care to fill in the details, see Lemmas \ref{closesequence} and \ref{JJ:theo} below. 
\end{remark}

\begin{remark}\label{rmk:more_gen}
It is natural to wonder if Theorem \ref{Th:mp} can be extended to more general domains. We 
conjecture that this is the case, although this remains an open problem. 

At least, in view of Remark \ref{rmk:normvslevinstrongconv}, we can relax the assumption 
of $\Omega$ star-shaped by requiring that $\Omega$ is such that \eqref{eq:main} does not admit 
positive solutions for $\lambda<0$. Actually, this is not true for every domain, as it fails 
e.g. in annuli.
 
Some further suggestions come from the recent paper \cite{JJS_absrtact}, which contains an abstract 
version of the monotonicity trick carrying information on the Morse index. In view of this, it may be 
possible to further weaken the assumptions of Theorem \ref{Th:mp}, only requiring that $\Omega$ 
is such that \eqref{eq:main} does not admit positive solutions having Morse index at most one 
(in $M_\rho$) for $\lambda<0$. The use of the Morse index to control bubbling phenomena in the context 
of normalized solutions was already exploited in \cite{ntvAnPDE}, to prove the existence of normalized 
unstable solutions.
\end{remark}

\begin{remark}\label{rmk:mfg_crit}
In the context of normalized solutions for ergodic Mean Field Games, a variational approach exploited 
in \cite{CCV23} shows that also in this framework, in the Sobolev critical case, the problem admits a 
mountain pass geometry and a local minimizer. The existence of a mountain pass solution remains an open problem, though.
\end{remark}

The paper is structured as follows. In Section \ref{sec:prel}, after a change of variable to move the parameter from the constraint to the equation (and the energy), we collect some preliminary results and 
tools, notably the versions we need of the monotonicity trick and of the analysis of the failure of the 
Palais-Smale property. As we mentioned, at the end of such section we fill the gap in the proof of 
\cite[Proposition 4.4]{MR3689156}. In Section \ref{sec:MP} we analyze the mountain pass geometry 
of our Sobolev critical variational problem. Finally, in Section \ref{sec:proof_main_res} we provide the proofs of the 
three theorems stated in this introduction.

\section{Notation and preliminary results}\label{sec:prel}

In this section, we give some notations and preliminary results which will be used in the proofs of our main results. 

In this paper, for a bounded domain $\Omega\subset \R^N$, $N\ge3$, 
we denote by $L^{p}(\Omega)$ the Lebesgue space with norm $\|\cdot\|_{p}$ and by
$H^{1}_{0}(\Omega)$ the usual Sobolev space with the norm $\|\cdot\|_{H^{1}_{0}}=\|\nabla \cdot\|_{2}$, while $D^{1,2}(\R^N)$ stands for the homogenous Sobolev space with norm
$\|\nabla \cdot\|_{2}$. Denoting with $S$ the best Sobolev constant of the embedding $D^{1,2}(\R^N)
\hookrightarrow L^{2^{*}}(\R^N)$, we have
\begin{equation}\label{ineq:sobolev}
S=\inf\limits_{u\in H^{1}_{0}(\Omega)\setminus \{0\}}\frac{\|\nabla u\|_2^{2}}{\|u\|_{2^{*}}^2}=\min\limits_{u\in D^{1,2}(\R^N)\setminus \{0\}}\frac{\|\nabla u\|_{2}^{2}}{\|u\|_{2^{*}}^2}.
\end{equation}

Turning to our problem, for convenience of calculation we  apply the transformation
\begin{equation}\label{eq:changeofvariables}
u=\frac{1}{\rho}U,\qquad \mu=\rho^{2^{*}-2},
\end{equation}
to convert problem \eqref{eq:main} into the following:
\begin{equation}\label{eq:main1}
\begin{cases}
-\Delta u = \lambda u + \mu |u|^{2^{*}-2}u  & \text{in }\Omega,\smallskip\\
\int_\Omega u^2\,dx = 1, \quad u=0  & \text{on }\partial\Omega,
\end{cases}
\end{equation}
so that solutions of \eqref{eq:main1} correspond to critical points of the energy functional $E_{\mu}:H^{1}_{0}(\Omega)\rightarrow \mathbb{R}$, 
\begin{equation}\label{func:main}
E_{\mu}(u)=\frac{1}{2}\int_{\Omega}|\nabla u|^{2}dx-\frac{\mu}{2^{*}}\int_{\Omega}|u|^{2^{*}}dx,
\end{equation}
on the $L^{2}$-sphere
\begin{equation}\label{manifold:1}
M=\left\{u\in H^{1}_{0}(\Omega): \int_{\Omega} u^{2}=1\right\}.
\end{equation}
Since most of the literature concerning \eqref{eq:main}, \eqref{eq:main1} is stated for non-normalized 
solutions (i.e. for given $\lambda\in\R$, without conditions on $\|u\|_2$) it is also convenient to 
introduce the action functional $I_{\lambda,\mu}:H^{1}_{0}(\Omega)\rightarrow \mathbb{R}$, 
\begin{equation}\label{func:auto}
I_{\lambda,\mu}(u)= E_\mu(u) -\frac{\lambda}{2}\int_{\Omega}|u|^{2}dx  = \frac{1}{2}\int_{\Omega}|\nabla u|^{2}dx-\frac{\lambda}{2}\int_{\Omega}|u|^{2}dx-\frac{\mu}{2^{*}}\int_{\Omega}|u|^{2^{*}}dx.
\end{equation}
In particular, since $M\subset H^1_0(\Omega)$ is an embedded Hilbert manifold, the tangential gradient 
$\nabla_M E_\mu(u) \in T_Mu$ can be identified with an element of $H^1_0(\Omega)$, so that
\begin{equation}\label{eq:def of nabla_M E}
\nabla_M E_\mu(u) = \nabla_{H^1_0} I_{\lambda,\mu}(u),
\qquad\text{where }
\lambda = \int_{\Omega}|\nabla u|^{2}dx-\mu\int_{\Omega}|u|^{2^{*}}dx.
\end{equation}

To proceed, we recall the results we need, in order to deal with candidate Mountain Pass critical levels.
Two classical tools in this direction are contained in the following two lemmas: the first one is 
essentially Theorem 4.5 in the book of Ghoussoub \cite{MR1251958}, which is needed to construct 
localized Palais-Smale sequences, in order to obtain positive solutions (we report the statement 
for the reader's convenience); the second one is an adaptation of Jeanjean's version \cite{JJ_BPS} 
of the so-called Struwe monotonicity trick (see \cite{MR970849,MR926524}).

\begin{lemma}[{\cite[Thm. 4.5, Rmk. 4.10]{MR1251958}}]\label{closesequence}
Let $M$ be a complete connected Hilbert manifold, $\varphi_*\in C^{1}(M,\mathbb{R})$ be a given 
functional, $K\subset M$ be compact and consider a nontrivial subset
\[
\Gamma\subset \{\gamma\subset M:\gamma \ \text{is compact }, \ K\subset \gamma\}
\]
which is invariant under deformations which decrease the value of $\varphi_*$ and leave $K$ fixed. 
Assume that $\bar \nu>0$ is such that 
\[
\max_{u\in K}\varphi_*(u) + \bar\nu < c_*:=
\inf\limits_{\gamma\in \Gamma}\max\limits_{u\in \gamma}\varphi_*(u).
\]
Then there exists $0<\bar\delta<\bar\nu$ such that
\begin{itemize}
\item for every $\eps^*,\delta_* \in (0,\bar\delta)$ with $\eps^*\le\delta^*$;
\item for every $\gamma_*\in\Gamma_{\eps_*}:=\{\gamma\in\Gamma: \max_{u\in \gamma}\varphi_*(u)\le 
c_* + \eps_*\}$;
\item for every closed set $F_*\subset M$ satisfying:
\begin{enumerate}
\item $F_*$ has nonempty intersection with every element of $\Gamma_{\eps_*}$, 
\item $\inf_{F_*} \varphi_* \ge c_*-\delta_*$;
\end{enumerate}
\end{itemize}
there exists $v_*\in M$ such that
\[
|\varphi_*(v_*)- c_*| + \|\nabla_M \varphi_*(v_*)\|+\dist(v_*,\gamma_*)+\dist(v_*,F_*)
\le 9\delta_* + 26 \sqrt{\delta_*}.
\]
\end{lemma}

\begin{lemma}\label{JJ:theo}
Let $E_\mu$ and $M$ be defined as in \eqref{func:main}, \eqref{manifold:1}, respectively,
and let $J\subset \mathbb{R}^{+}$ be an interval. Let us assume that there are two points
$w_{0}$, $w_{1}$ in $M$ such that, setting
\[
\Gamma=\{\gamma \in C([0,1],M), \ \gamma(0)=w_{0}, \ \gamma(1)=w_{1}\},
\]
we have, for some $\bar\nu>0$ and for every $\mu\in J$,
\[
c_{\mu}:=\inf\limits_{\gamma\in \Gamma}\max\limits_{t\in[0,1]}E_{\mu}(\gamma(t))> \max\{E_{\mu}(w_{0}),E_{\mu}(w_{1})\} + \bar \nu.
\]
Then:
\begin{enumerate}
\item for every $\mu\in J$, the map $\mu\mapsto c_\mu$ is continuous from the left;
\item for almost every $\mu\in J$, the map $\mu\mapsto c_\mu$ is differentiable, with
derivative $c'_\mu$.
\end{enumerate}
Moreover, for every $\mu$ such that $c'_\mu$ exists, there exists a sequence $(v_{n})_{n}\subset M$,
such that, as $n\to+\infty$,
\begin{enumerate}
\item\label{PS1} $E_{\mu}(v_{n})\rightarrow c_{\mu}$;
\item\label{PS2} $\|\nabla_M E_{\mu}(v_{n})\|\rightarrow 0$;
\item\label{positivejnjn} if $w_{0},w_{1}\geq0$, $\|(v_n)^{-}\|_{H^1_0}\to 0$; and
\item\label{boundjnjn} $\|\nabla v_{n}\|_{2}^{2} \le 2c_{\mu}-2c^{\prime}_{\mu}\mu
+6\mu +o_n(1)$.
\end{enumerate}
\end{lemma}

\begin{proof}
Except for property \ref{positivejnjn}, this lemma can be proved with minor modifications as in the 
paper by Jeanjean, \cite[Prop. 2.1, Lemma 2.3]{JJ_BPS}, which was set in a Hilbert space (see also 
\cite{JJS_absrtact}). We sketch such proof here because the almost positivity in \ref{positivejnjn} 
requires some care in the application of Lemma \ref{closesequence}, and also because 
we need the explicit bound provided in \ref{boundjnjn}. More precisely, we will not apply Lemma 
\ref{closesequence} to a fixed functional $E_\mu$, but to a sequence of approximating functionals.

The starting point is that, for every fixed $u\in M$, the map $J\ni\mu\mapsto E_\mu(u)$ is decreasing. 
From this, arguing as in \cite[Lemma 2.3]{JJ_BPS}, one infers that the map $\mu\mapsto c_\mu$ is  
continuous from the left on $J$; moreover, such map is non-increasing, and thus differentiable for 
almost every $\mu\in J$. In the following, we fix a differentiability point $\mu$ and we 
consider a strictly increasing sequence $(\mu_{n})_{n}\subset J$ such that $\mu_{n}\rightarrow \mu$. 

To construct the required sequence $(v_n)_n$ we will apply (infinitely many times) Lemma \ref{closesequence}, with $K=\{w_0,w_1\}$, $\Gamma$ as above, 
\[
\varphi_*=E_{\mu_n},\quad c_*=c_{\mu_n},\quad \eps_*=\eps_n := \mu - \mu_n,\quad \delta_*=\delta_n := c_{\mu_n} - c_\mu + \mu - \mu_n\searrow 0,
\]
as $n\to+\infty$, by left continuity of the map $\mu\mapsto c_\mu$; in particular,  
$\bar\delta$ can be fixed only depending on $J$ and $\bar\nu$, so that, for $n$ large, $0<\eps_n<\delta_n<\bar\delta$. With this choice, we have
\begin{equation*}%\label{eq:near0}
\Gamma_{\eps_n} := \{\gamma \in\Gamma : \max_{t}E_{\mu_n}(\gamma(t)) \le c_{\mu_n}+\mu - \mu_n \}\end{equation*}
and, in case $w_{0},w_{1}\geq0$, we can choose $\gamma_*=\gamma_n$ as a path of nonnegative 
functions: indeed, in such a case $|w_{0}|=w_{0}$ and $|w_{1}|=w_{1}$ and, since $E_{\mu}$ and $M$ are 
even, $\gamma_{n}\in\Gamma_n$ implies $|\gamma_{n}|\in\Gamma_n$, with the same value of $E_\mu$. 
Finally, we choose
\[
F_*=F_n:=\overline{\left\{u \in M : 
\begin{array}{c}
u = \tilde\gamma(\tilde t),\text{ for some }\tilde\gamma\in\Gamma_{\eps_n},\ \tilde t\in(0,1), \text{ and }\\
 E_{\mu}(u) \ge c_\mu-\eps_n 
\end{array}
\right\}}
\]
so that the first assumption about $F_n$ in Lemma \ref{closesequence} is trivially satisfied, while
\[
\inf_{F_*}\varphi_*=\inf_{F_n}E_{\mu_n} \ge \inf_{F_n}E_{\mu} \ge  c_\mu-\eps_n = c_{\mu_n} - \delta_n.
\]

As a consequence, all the assumptions of Lemma \ref{closesequence} are satisfied and 
we infer the existence of a sequence $(v_{n})_{n}$ such that, as $n\to+\infty$,
\begin{equation}\label{eq:prePS_n}
E_{\mu_n}(v_{n})\to c_{\mu}, \ \ \ \|\nabla_M E_{\mu_n}(v_{n})\|\to 0 \ \ \ \text{and   } \ \ \|v_{n}-|\gamma_{n}(t_{n})|\| \to 0 \ \ \text{for some } t_{n}\in[0,1],
\end{equation}
and in particular \ref{positivejnjn} follows by the choice of $(\gamma_n)_n$. Moreover, 
the sequence $(v_n)_n$ also satisfies, as $n\to+\infty$,
\begin{equation}\label{eq:near}
\dist(v_{n},u_{n})\to 0,\qquad\text{for some }u_n\in F_n.
\end{equation}
We claim that this implies \ref{boundjnjn}. Indeed, using the  
definition of $F_n$ and $\Gamma_{\eps_n}$, we can choose $u_n=\tilde\gamma_n(\tilde t_n)$, for some 
$\tilde\gamma_n\in \Gamma_{\eps_n}$, $\tilde t_{n}\in(0,1)$, so that
\begin{equation}\label{eq:near2}
\text{both } E_{\mu_n}(u_n) \le c_{\mu_n}+\mu - \mu_n\qquad
\text{and } E_{\mu}(u_n) \ge c_{\mu}-(\mu - \mu_n).
\end{equation}
But this implies
\begin{equation*}
\frac{E_{\mu_{n}}(u_n)-E_{\mu}(u_n)}{\mu-\mu_{n}}\leq \frac{c_{\mu_{n}}+(\mu-\mu_{n})-c_{\mu}+(\mu-\mu_{n})}{\mu-\mu_{n}}=\frac{c_{\mu_{n}}-c_{\mu}}{\mu-\mu_{n}}+2.
\end{equation*}
Since $c_{\mu}^{\prime}$ exists, there is $n(\mu)\in \mathbb{N}$ such that $\forall n\geq n(\mu)$
\begin{equation*}%\label{cmu}
-c_{\mu}^{\prime}-1\leq\frac{c_{\mu_{n}}-c_{\mu}}{\mu-\mu_{n}}\leq-c_{\mu}^{\prime}+1,
\end{equation*}
which implies that, for $ n\geq n(\mu)$,
\begin{equation*}
\frac{E_{\mu_{n}}(u_n)-E_{\mu}(u_n)}{\mu-\mu_{n}}\leq-c_{\mu}^{\prime}+3.
\end{equation*}
Thus,
\begin{equation*}
\frac{1}{2^{*}}\|u_n\|_{2^{*}}^{2^{*}}=\frac{E_{\mu_{n}}(u_n)-E_{\mu}(u_n)}{\mu-\mu_{n}}
\leq-c_{\mu}^{\prime}+3,
\end{equation*}
which means that
\begin{equation*}
\frac{1}{2}\|\nabla u_n\|_{2}^{2}= E_{\mu}(u_n)+\frac{\mu}{2^{*}}\|\gamma_{n}(t)\|_{2^{*}}^{2^{*}}
\leq c_{\mu_{n}}+o_n(1)+\mu(-c_{\mu}^{\prime}+3),
\end{equation*}
and finally
\begin{equation*}
\|\nabla u_n\|_{2}^{2} \le 2c_{\mu}-2c^{\prime}_{\mu}\mu+6\mu +o_n(1).
\end{equation*}
Recalling \eqref{eq:near}, we obtain that the claim follows, and $(v_n)_n$ satisfies \ref{boundjnjn}.

Finally, since
\[
E_\mu(u) = E_{\mu_n}(u) - \frac{\mu-\mu_n}{2^*}\int_{\Omega}|u|^{2^{*}}dx
\]
and $(v_n)_n$ is bounded, we infer that \eqref{eq:prePS_n} implies properties \ref{PS1} and \ref{PS2}, 
and the lemma follows.
\end{proof}
\begin{remark}\label{rmk:trick_for_p}
The content of Lemma \ref{JJ:theo} holds true also in case $E_\mu$ is replaced with the Sobolev subcritical energy functional
\begin{equation}\label{eq:subcrit_energy}
\tilde E_{\mu}(u)=\frac{1}{2}\int_{\Omega}|\nabla u|^{2}dx-\frac{\mu}{p}\int_{\Omega}|u|^{p}dx,
\end{equation}
with $2+\frac4N < p < 2^*$. In particular, if $\tilde E_\mu$ enjoys a common mountain pass geometry 
for $\mu$ in some interval $J$, then it admits a bounded Palais-Smale sequence at the mountain pass 
level for almost every $\mu\in J$.
\end{remark}

Next, to investigate the convergence of Palais-Smale sequences, we recall the following result
by Struwe for $I_{\lambda,\mu}(u)$, which is essentially a version of \cite[Ch. III, Thm. 3.1]{Struwebook}.
\begin{lemma}\label{split}
Suppose $\Omega$ is a smoothly bounded domain in $\mathbb{R}^{N}$, $N\geq3$, $\lambda\in\R$, $\mu>0$.
Let moreover $(u_{n})_{n}$ be a Palais-Smale sequence for $I_{\lambda,\mu}(u)$ in $H^{1}_{0}(\Omega)$. Then, up to a subsequence, $(u_{n})_n$ weakly converges to a function $u^{0}\in H^{1}_{0}(\Omega)$ and if the convergence is not strong then there exist an integer $k\geq1$, $k$ nontrival solutions $u^{j}\in H^{1}_{0}(\Omega)$, $1\leq j\leq k$, to the "limiting problem",
\begin{equation}\label{pro:lim}
-\Delta u = \mu |u|^{2^{*}-2}u \ \ \ \ \mbox{in } \mathbb{R}^{N}
\end{equation}
such that
\[
\left\|u_{n}-\left(u^{0}+\sum\limits_{j=1}\limits^{k}u^{j}_{n}\right)\right\|_{D^{1,2}(\R^N)}\to0,
\]
as $n\to+\infty$, where
\[u^{j}_{n}=(R_{n}^{j})^{\frac{N-2}{2}}u^{j}(R_{n}^{j}(x-x_{n}^{j})) \ \ \ 1\leq j \leq k, \ n\in \N,\]
for suitable points $x_{n}^{j}\in \Omega$ and positive parameters $R_{n}^{j}\rightarrow+\infty$ as $n\rightarrow+\infty$. Moreover, we have
\[
\begin{split}
\|\nabla u_{n}\|_{2}^{2}&=\|\nabla u^{0}\|_{2}^{2}+\sum\limits_{j=1}\limits^{k}\|\nabla u^{j}\|_{2}^{2}+o(1),\\
\| u_{n}\|_{2^*}^{2^*}&=\| u^{0}\|_{2^*}^{2^*}+\sum\limits_{j=1}\limits^{k}\| u^{j}\|_{2^*}^{2*}+o(1),\\
\| u_{n}\|_{2}^{2}&=\| u^{0}\|_{2}^{2}+o(1).
\end{split}
\]
\end{lemma}
The above result concerns the case with fixed $\lambda$. Here we need the following version, which
is adapted to the normalized Palais-Smale sequences.
\begin{corollary}\label{cor:split}
Let $(u_{n})_n$ be a Palais-Smale sequence for $E_{\mu}$ on $M$:
\begin{enumerate}
\item $(u_n)_n \subset M$,
\item $E_\mu(u_n) \to c \in \R$,
\item $\|\nabla_M E_\mu(u_n)\| = \|\nabla_{H^1_0} I_{\lambda_n,\mu}(u_n)\| \to 0$,
\end{enumerate}
as $n\to+\infty$, and assume that the sequence
\begin{equation}\label{eq:lambda_def}
\lambda_n = \int_{\Omega}|\nabla u_n|^{2}dx-\mu\int_{\Omega}|u_n|^{2^{*}}dx
\end{equation}
is bounded.

Then there exist $\lambda^0\in\R$, $u^0\in M$ such that, up to subsequences,
$\lambda_n\to\lambda^0$ and either
\[
u_{n}\rightarrow u^{0} \ \ \text{in } H^{1}_{0}(\Omega)
\]
or
\[
u_{n}\rightharpoonup u^{0},\ u_{n}\not\rightarrow u^{0} \ \ \text{in } H^{1}_{0}(\Omega)
\]
with
\begin{equation}\label{eq:split_E}
\liminf E_{\mu}(u_{n})\geq E_{\mu}(u^{0})+ \frac{1}{N} S^{N/2}\mu^{1-\frac{N}{2}},
\end{equation}
and
\begin{equation}\label{eq:split_norm}
 \liminf \|\nabla u_{n}\|_{2}^{2}\geq\|\nabla u^{0}\|_{2}^{2}+S^{N/2}\mu^{1-\frac{N}{2}}.
\end{equation}
In both cases, $u^0$ is a solution of \eqref{eq:main1}, with $\lambda = \lambda^0$.
\end{corollary}
\begin{proof}
We notice that, for every $u\in H^1_0(\Omega)$,
\[
I_{\lambda^0,\mu} (u) = I_{\lambda_n,\mu} (u) + \frac{\lambda^0-\lambda_n}{2}\int_{\Omega}|u|^{2}dx.
\]
Since $\lambda_n \to \lambda^0$, we deduce that $(u_n)_n$ is a Palais-Smale sequence for
$I_{\lambda^0,\mu}$ in $H^1_0(\Omega)$. Then, we can apply Lemma \ref{split} and conclude, using the
compact embedding of $H^1_0(\Omega)$ into $L^2(\Omega)$, and recalling that every solution $v$
of \eqref{pro:lim} satisfies
\[
E_{\mu}(v) = \frac{1}{N} \|\nabla v\|^{2}_{L^2(\R^N)} \ge \frac{1}{N} S^{N/2}\mu^{1-\frac{N}{2}}
\]
(see e.g. \cite[Ch. III, Rmk. 3.2, $(3^\circ)$]{Struwebook}.)
\end{proof}

Finally, we will use repeatedly the following elementary lemma. 
\begin{lemma}\label{fmax}Define $f(s)=\dfrac{1}{2}As-\dfrac{1}{2^{*}}Bs^{\frac{2^{*}}{2}}$, let $s,\bar{s}\in \mathbb{R}^{+}$ satisfy $\max\limits_{s>0}f(s)=f(\bar{s})$. Then,
\[\max\limits_{s>0}f(s)=f(\bar{s})= \frac{1}{N}\frac{A^{\frac{N}{2}}}{B^{\frac{N}{2}-1}}\]
where $\bar{s}=\left(\dfrac{A}{B}\right)^{\frac{N}{2}-1}$.
\end{lemma}
\begin{proof}
By calculation, we obtain
\[f^{\prime}(s)= \frac{1}{2}A-\frac{1}{2}Bs^{\frac{2^{*}}{2}-1}.\]

Therefore, let $f^{\prime}(s)=0$, we have $s^{\prime}=(\frac{A}{B})^{\frac{N}{2}-1}$. If $0<s<s^{\prime}$, we have $f^{\prime}(s)>0$. If $0<s^{\prime}<s$, we have $f^{\prime}(s)<0$. Thus, $s^{\prime}=\bar{s}$ and
\begin{equation*}
%    \begin{aligned}\displaystyle
f(\bar{s})=\frac{1}{2}A\bar{s}-\frac{1}{2^{*}}B\bar{s}^{\frac{2^{*}}{2}}= \frac{1}{2}\frac{A^{\frac{N}{2}}}{B^{\frac{N}{2}-1}}-\frac{1}{2^{*}}\frac{A^{\frac{N}{2}}}{B^{\frac{N}{2}-1}}
= \frac{1}{N}\frac{A^{\frac{N}{2}}}{B^{\frac{N}{2}-1}}. \qedhere
%    \end{aligned}
\end{equation*}
\end{proof}

To conclude this section, we can fill the gap in the mountain pass result stated in \cite{MR3689156}.
\begin{proof}[Proof of {\cite[Proposition 4.4]{MR3689156}} completed]
It was already shown in \cite{MR3689156} that $\tilde E_\mu$, defined as in \eqref{eq:subcrit_energy}, 
has the mountain pass geometry for $\mu\in(0,\mu^*)$, for a suitable $\mu^*$. Then Remark 
\ref{rmk:trick_for_p} applies and we obtain the existence of a bounded Palais-Smale sequence $(v_n)_n$ 
for $\tilde E_\mu$ on $M$, at the mountain pass level $c_\mu$, for every $\mu\in F$, where $F\subset (0,\mu^*)$ is such that $(0,\mu^*)\setminus F$ is negligible. Then $v_n \rightharpoonup \bar v$ weakly in $H^1_0(\Omega)$ and, since $p<2^*$, we obtain, up to subsequences
\[
\|v_n\|_2 \to \|\bar v\|_2,\qquad
\|v_n\|_p \to \|\bar v\|_p,\qquad
\|\nabla v_n\|_2^2 - \mu\|v_n\|_p^p =:\lambda_n \to \bar \lambda \in\R.
\]
Since $\nabla_M\tilde E_\mu(v_n) \to 0$, we infer that $\|\nabla v_n\|_2 \to \|\nabla \bar v\|_2$, 
and thus $v_n \to \bar v = \bar v_\mu$ strongly in $H^1_0(\Omega)$, where  $\bar v_\mu$ is a 
normalized solution at the mountain pass level $c_\mu$, for every $\mu\in F$. Finally, let 
$\hat\mu\in (0,\mu^*)\setminus F$, and $(\mu_k)_k\subset F$, with $\mu_k\to\hat\mu$. Although 
in principle $(\bar v_{\mu_k})_k$ does not need to be bounded, on the other hand, it has Morse index 
uniformly bounded (it is made of mountain pass solutions) and therefore the blow-up analysis contained 
in \cite[Section 2]{MR3689156} applies, in particular \cite[Corollary 2.15]{MR3689156}; then, since 
$\mu_k\to\hat\mu\neq 0$, we deduce that $(\bar v_{\mu_k})_k$ is bounded in $H^1_0(\Omega)$. But then we 
can repeat the argument above, obtaining that $\bar v_{\mu_k} \to \bar v_{\hat\mu}$ strongly in 
$H^1_0(\Omega)$, and  $\bar v_{\hat\mu}$ is a 
normalized solution at the mountain pass level $c_{\hat\mu}$, concluding the proof.
\end{proof}

\section{Mountain pass structure}\label{sec:MP}
In order to investigate the geometric structure of $E_{\mu}$, we define the sets
\begin{equation}\label{sets:geo}
\mathcal{B}_{\alpha}=\left\{u\in M : \int_{\Omega}|\nabla u|^{2}dx<\alpha\right\}, \ \ \ \ \ \ \mathcal{U}_{\alpha}=\left\{u\in M : \int_{\Omega}|\nabla u|^{2}dx=\alpha\right\}.
\end{equation}
It is well known that $\Bcal_\alpha$ is non-empty if and only if
$\alpha\ge\lambda_{1}(\Omega)$, in which case
\begin{equation}\label{eq:ph1}
\varphi_1 \in \mathcal{B}_{\alpha},
\end{equation}
where $\lambda_{1}(\Omega)$ is the first eigenvalue and $\varphi_{1}>0$,
$\|\varphi_{1}\|_{2}^{2}=1$ is the first eigenfunction of the problem
\begin{equation}\label{eq:ph}
\begin{cases}
-\Delta \varphi_{1} = \lambda_{1}(\Omega) \varphi_{1} & \text{in }\Omega\\
\varphi_1=0 & \text{on }\partial\Omega.
\end{cases}
\end{equation}

We will find that $E_{\mu}$ possesses the following geometric structure.
\begin{lemma}\label{geo-struc}
Assuming that
\begin{equation}\label{eq:cond_baralpha}
\mu<\mu^*:= \left(\dfrac{2S^{N/2}}{N\lambda_{1}(\Omega)}\right)^{\frac{2}{N-2}},
\end{equation}
let us define
\begin{equation}\label{eq:alphabar}
\bar{\alpha}=\bar\alpha(\mu)=S^{N/2}\mu^{1-\frac{N}{2}}.
\end{equation}
Then
\begin{equation}
\inf\limits_{\mathcal{B}_{\bar\alpha}} E_{\mu}<  \frac{1}{N}S^{N/2}\mu^{1-\frac{N}{2}}
\le \inf\limits_{\mathcal{U}_{\bar\alpha}} E_{\mu}.
\end{equation}
\end{lemma}

\begin{proof}
Using \eqref{func:main} and \eqref{ineq:sobolev} we obtain
\begin{equation*}
\begin{aligned}\displaystyle
\inf\limits_{\mathcal{U}_{\alpha}} E_{\mu}&\geq \frac{1}{2}\int_{\Omega}|\nabla u|^{2}dx-\frac{\mu}{2^{*}}\int_{\Omega}|u|^{2^{*}}dx\\
&\geq \frac{1}{2}\int_{\Omega}|\nabla u|^{2}dx-\frac{\mu}{2^{*}S^{\frac{2^{*}}{2}}}(\int_{\Omega}|\nabla u|^{2}dx)^{\frac{2^{*}}{2}}\\
&= \frac{1}{2}\alpha -\frac{\mu}{2^{*}S^{\frac{2^{*}}{2}}}(\alpha)^{\frac{2^{*}}{2}}=:f(\alpha).
\end{aligned}
\end{equation*}
By Lemma \ref{fmax}, we have that $\max_{\alpha>\lambda_1(\Omega)} f(\alpha) = f (\bar\alpha)$, with
\[
\bar{\alpha}=S^{N/2}\mu^{1-\frac{N}{2}}>\lambda_1(\Omega)
\]
by \eqref{eq:cond_baralpha}. Thus,
\[\inf\limits_{\mathcal{U}_{\bar\alpha}} E_{\mu}\geq \frac{1}{N}S^{N/2}\mu^{1-\frac{N}{2}}.\]

On the other hand, by \eqref{eq:ph1} we infer
\[\inf\limits_{\mathcal{B}_{\bar\alpha}} E_{\mu}\leq E_{\mu}(\varphi_{1})\leq\frac{1}{2}\lambda_{1}(\Omega).\]
Using $\eqref{eq:cond_baralpha}$ again, we obtain
\[
\inf\limits_{\mathcal{B}_{\bar\alpha}} E_{\mu}\leq\frac{1}{2}\lambda_{1}(\Omega)<\frac{1}{N}S^{N/2}\mu^{1-\frac{N}{2}}\leq\inf\limits_{\mathcal{U}_{\bar\alpha}} E_{\mu}.
\qedhere
\]
\end{proof}

Based on the previous lemma, we define two different levels of the energy, depending on
$\mu$, which will be proved to be critical in the following sections.

First, we define the candidate local minimum level as
\begin{equation}\label{eq:minimumlevel}
m_{\mu}:=\inf\limits_{\mathcal{B}_{\bar\alpha}}E_{\mu},
\end{equation}
By Lemma \ref{geo-struc}, we know that
\begin{equation}\label{level:min}
m_{\mu}<\frac{1}{N}S^{N/2}\mu^{1-\frac{N}{2}}.
\end{equation}

On the other hand, exploiting the local minimum geometry introduced above, we are going to
define a second candidate critical value of mountain pass type.
To this aim, we need to find two functions $w_0$, $w_1$ in $M$ with the following properties:
\begin{equation}\label{eq:w0_inside}
w_0\in \Bcal_{\bar \alpha},\qquad E_\mu(w_0) < \frac{1}{N}S^{N/2}\mu^{1-\frac{N}{2}},
\end{equation}
\begin{equation}\label{eq:w1_outside}
w_1\not\in \Bcal_{\bar \alpha},\qquad E_\mu(w_1) < 0.
\end{equation}
Then the mountain pass structure will be obtained using paths joining $w_0$ and $w_1$. The
main obstruction to do this is that, while $w_0$ can be obtained independently on $\mu$, to
find $w_1$ we have to work on intervals of $\mu$ bounded away from zero.

Let us assume without loss of generality that
$0\in \Omega$. For every
$\eps>0$, let us introduce the function $u_\eps \in H^1_0(\Omega)$ defined as
\begin{equation}\label{eq:def_u_eps}
u_{\varepsilon}(x)=\frac{\eta[N(N-2)\varepsilon^{2}]^{^{\frac{N-2}{4}}}}{[\varepsilon^{2}+ |x|^{2}]^{\frac{N-2}{2}}},
\end{equation}
where $\eta\in C^{\infty}_{0}(\Omega)$, $0\le \eta \le 1$, is a fixed cut-off function
such that $\eta\equiv1$ in some neighborhood of $0$. Then, we define
\[v_{\varepsilon}:=\frac{u_{\varepsilon}}{\|u_{\varepsilon}\|_{2}},\]
which implies that $\|v_{\varepsilon}\|_{2}=1$, namely, $v_{\varepsilon}\in M$, for every $\eps>0$.

Now, using the estimates provided by Struwe \cite[page 179]{Struwebook} (or those by
Brezis and Nirenberg, \cite[eqs. (1.13), (1.29)]{BrezisNirenberg}), we have
\begin{equation}\label{est:struwe}
    \begin{aligned}\displaystyle
\|\nabla u_{\varepsilon}\|_{2}^{2}&=S^{N/2}+O(\varepsilon^{N-2}),\\
\|u_{\varepsilon}\|_{2^{*}}^{2^{*}}&=S^{N/2}+O(\varepsilon^{N}),\\
\|u_{\varepsilon}\|_{2}^{2}&=
\begin{cases}
c \varepsilon^{2}+O(\varepsilon^{N-2}), & \text{if }N\geq5,\\
c \varepsilon^{2}|\ln \varepsilon|+O(\varepsilon^{2}), & \text{if }N=4,\\
c \varepsilon+O(\varepsilon^{2}), & \text{if } N=3,
\end{cases}
    \end{aligned}
\end{equation}
where $c$ denotes a strictly positive constant (depending on $N$).
\begin{remark}\label{rmk:why_R}
We stress the fact that the above estimates, in particular the one for 
$\|u_{\varepsilon}\|_{2^{*}}^{2^{*}}$, are based on the fact that $\eta\equiv1$ in some 
neighborhood of $0$ (and they actually depend on the choice of such neighborhood).
\end{remark}

\begin{lemma}\label{lem:w0}
There exists $0<\mu^{**}\le \mu^*$ such that, for every $0<\mu<\mu^{**}$ we have that
$w_0=v_1$ satisfies \eqref{eq:w0_inside}.
\end{lemma}
\begin{proof}
The proof is trivial since $w_0$ does not depend on $\mu$, while both $\bar\alpha(\mu)$ and
$\frac{1}{N}S^{N/2}\mu^{1-\frac{N}{2}}$ increase to $+\infty$ as $\mu\to0$.
\end{proof}
\begin{remark}\label{rmk:better_w0}
Actually, with more work, in the previous lemma it is possible to choose $\mu^{**}$ as
near as $\mu^*$ as we want.
Indeed, one can see that $v_\eps\to \eta/\|\eta\|_2$ in $H^1_0(\Omega)$ as $\eps\to+\infty$.
As a consequence, one can take $0\in\Omega$ to be a maximum point of $\varphi_1$ and $\eta$ a small 
perturbation of $\varphi_1/\|\varphi_1\|_\infty$, locally constant near $0$: in this way, it is possible to show that $w_0=v_\eps$
satisfies \eqref{eq:w0_inside}, for every $\mu\leq \mu^{**}<\mu^*$, as long as $\eps$ is fixed
large enough. On the other hand, the rigorous estimates to prove this are quite long, and
in the following we will need to lower $\mu^{**}$ to fulfill further conditions. For these
reasons, we do not pursue this strategy here.
\end{remark}
\begin{lemma}\label{lem:w1}
Let $\mu^{**}$ be defined as in Lemma \ref{lem:w0}, and let $0<\mu_0\le \mu^{**}$. There
exists $\eps_1>0$, dependent on $\mu_0$, such that $w_1=v_{\eps_1}$ satisfies
\eqref{eq:w1_outside} for every $\frac{\mu_0}{2}<\mu<\mu_0$.
\end{lemma}
\begin{proof}
The proof follows using the asymptotic estimates \eqref{est:struwe}. Indeed, we have
\[
\|\nabla v_\eps\|_2^2 = O(\|u_\eps\|_2^{-2})\to+\infty \qquad\text{as }\eps\to0.
\]
As a consequence, if $\eps$ is sufficiently small, we obtain
\[
\|\nabla v_\eps\|_2^2 > \bar\alpha (\mu_0/2),
\]
so that $v_\eps \not \in \Bcal_{\bar \alpha (\mu)}$ for every $\mu>\mu_0/2$.

Analogously
\[
E_\mu(v_\eps) = O(\|u_\eps\|_2^{-2}) - \mu O(\|u_\eps\|_2^{-2^*})
\le O(\|u_\eps\|_2^{-2}) - \frac{\mu_0}{2} O(\|u_\eps\|_2^{-2^*}),
\]
where all the ``big O'' terms are independent on $\mu$. The lemma follows as $\|u_\eps\|_2\to0$ as $\eps\to0$.
\end{proof}

Letting $w_0$ and $w_1$ be defined as in Lemmas \ref{lem:w0} and \ref{lem:w1} respectively,
by \eqref{eq:w0_inside} and \eqref{eq:w1_outside}, we define in a standard way the mountain pass value
\begin{equation}\label{level:mp}
c_{\mu}:=\inf\limits_{\gamma\in\Gamma} \sup\limits_{t\in[0,1]}E_{\mu}(\gamma(t)),
\qquad\text{where }
\Gamma:=\{\gamma\in C([0,1], M) : \gamma(0)=w_0, \gamma(1)=w_1\}.
\end{equation}
In particular, by \eqref{eq:w0_inside}, \eqref{eq:w1_outside}, we have that, for any
$\gamma\in\Gamma$, $\gamma([0,1])$ intersects $\Ucal_{\bar\alpha}$.
By Lemma \ref{geo-struc} and the definition of $c_{\mu}$, we infer that
\begin{equation}\label{eq:c_mu_from_below}
m_{\mu}<\frac{1}{N}S^{N/2}\mu^{1-\frac{N}{2}}\leq c_{\mu},
\end{equation}
which implies in particular that $c_{\mu}\neq m_{\mu}$; as a consequence, if we prove that they are both critical levels, then problem \eqref{eq:main} admits two different solutions.
\begin{remark}\label{rmk:c_mu not depend on mu0}
Notice that, in principle, the class of paths $\Gamma$ depends on the choice of $\mu_0$, via the
definition of $w_1$ given in Lemma \ref{lem:w1}. As a matter of fact, using \eqref{eq:c_mu_from_below}
it is not difficult to prove that actually $c_\mu$ only depends on $\mu\in(0,\mu^{**})$. Indeed,
if $\mu \in (\mu_0/2,\mu_0) \cap (\mu_0'/2,\mu_0')$, with (say) $0<\mu_0<\mu_0'<\mu^{**}$, then
we can connect $w_1=v_{\eps_1}$ and $w_1' = v_{\eps_1'}$ with the arc $\sigma = \{v_\eps:\eps_0\le
\eps \le\eps_0'\}$, and $E_\mu<0$ on $\sigma$, at least for $\mu^{**}$ small enough. Since every element of $\Gamma'$ corresponds to an
element of $\Gamma$, and vice versa, by juxtaposition with $\sigma$, we obtain that
\[
0 < \inf\limits_{\gamma\in\Gamma} \sup\limits_{t\in[0,1]}E_{\mu}(\gamma(t)) =
\inf\limits_{\gamma'\in\Gamma'} \sup\limits_{t\in[0,1]}E_{\mu}(\gamma'(t)).
\]
\end{remark}

As a next step, we need a sharp estimate from above of $c_\mu$.
\begin{lemma}\label{energy:mp}
Let $N\ge3$, and let us assume, in case $N=3$, that $0<R<R_\Omega$, i.e. 
\[
\overline{B}_R\subset\Omega.
\]

Then, for a possibly smaller value of $\mu^{**}$, there exists $C>0$, not depending on $\mu$ (but 
depending on $R$ if $N=3$), such that, for every $0<\mu<\mu^{**}$
\begin{equation}\label{eq:mpfromabove}
c_{\mu}\leq
g(\mu):= \frac{1}{N}S^{N/2}\mu^{1-\frac{N}{2}}+ h(\mu),
\end{equation}
where
\begin{equation}\label{eq:mpfromabove_h}
    \begin{aligned}\displaystyle
h(\mu):=
\begin{cases}
C\mu^{\frac{(N-2)(N-4)}{4}}, \ \ \ &\mbox{if} \ N\geq5,
\smallskip\\
C|\ln \mu|^{-1}, \ \ \ &\mbox{if} \ N=4,
\smallskip\\
\frac{1}{4}\lambda_{1}(B_{R})+C\mu^{1/2}, \ \ \ &\mbox{if} \ N=3.
\end{cases}
    \end{aligned}
\end{equation}
\end{lemma}

\begin{proof}
Let $0<\mu<\mu^{**}$ and let us choose $\mu_0\le \mu^{**}$ in such a way that $\frac{\mu_0}{2}
<\mu<\mu_0$. With this choice of $\mu_0$, let $w_1 = v_{\eps_1}$ be defined as in Lemma \ref{lem:w1},
and let us define the maps
\[
t \mapsto \hat\eps,\ \ \hat\eps(t)=1 - (1-\eps_1)t,
\qquad\text{and} \qquad
\hat \gamma : [0,1] \to M,\ \ \hat\gamma(t) = v_{\hat\eps(t)}.
\]
By definition, we obtain that $\hat \gamma \in \Gamma$ (see also Remark \ref{rmk:c_mu not depend on mu0}), so that
\[
 c_{\mu}\leq \max_{0\le t\le 1} E_{\mu}(\hat{\gamma}(t)) = \max_{\eps_1\le \eps\le 1} E_\mu (v_\eps) \le
 \sup_{0<\eps\le1} E_\mu (v_\eps).
\]
Defining $a(\varepsilon):=\|u_{\varepsilon}\|_{2}^{2}$, by \eqref{est:struwe}, we divide the proof
into the following three cases (as the strategy is the same, we provide the full details only for the first case, and we just sketch the other two).

\underline{Case 1: $N\geq5$}. By \eqref{est:struwe}, we have
\begin{equation*}
    \begin{aligned}\displaystyle
 c_{\mu}\leq \sup\limits_{0<\varepsilon\le 1}E_{\mu}\left(\frac{u_{\varepsilon}}{\|u_{\varepsilon}\|_{2}}\right)&=\sup\limits_{0<\varepsilon\le 1}\left(\frac{1}{2}\frac{\|\nabla u_{\varepsilon}\|^{2}_{2}}{\| u_{\varepsilon}\|^{2}_{2}}-\frac{\mu}{2^{*}}\frac{\| u_{\varepsilon}\|^{2^{*}}_{2^{*}}}{\| u_{\varepsilon}\|^{2^{*}}_{2}}\right)\\
 &=\sup\limits_{0<\varepsilon\le 1}\left[\frac{1}{2}\frac{S^{N/2}+O(\varepsilon^{N-2})}{a(\varepsilon)}-\frac{\mu}{2^{*}}\frac{S^{N/2}+O(\varepsilon^{N})}{(a(\varepsilon))^{2^{*}/2}}\right]\\
 &=\sup\limits_{0<\varepsilon\le 1}\left[S^{N/2}\left(\frac{1}{2}\frac{1}{a(\varepsilon)}-\frac{\mu}{2^{*}}\frac{1}{(a(\varepsilon))^{2^{*}/2}}\right)+O(\varepsilon^{N-4})\right]\\
 &\leq S^{N/2} \sup\limits_{0<\varepsilon\le 1}\left[\frac{1}{2}\frac{1}{a(\varepsilon)}-\frac{\mu}{2^{*}}\frac{1}{(a(\varepsilon))^{2^{*}/2}}+C\varepsilon^{N-4}\right].
    \end{aligned}
\end{equation*}
Now, we can define $f(t):=\dfrac{1}{2}t-\dfrac{\mu}{2^{*}}t^{2^{*}/2}$ so that, by Lemma \ref{fmax}, $\max\limits_{t>0} f(t)=f(\mu^{1-\frac{N}{2}})$. Under this notation,
\begin{equation}\label{eq:cmuabove1}
c_{\mu}\leq S^{N/2}  \max\left[ \sup\limits_{A_1}f\left(\frac{1}{a(\varepsilon)}\right)+ C\eps^{N-4},\sup\limits_{A_2}f\left(\frac{1}{a(\varepsilon)}\right)+ C\eps^{N-4}\right]
\end{equation}
where
\[
A_1 = (0,1] \cap \{\eps:a(\eps)\le 2\mu^{\frac{N}{2}-1}\}
\qquad \text{and} \qquad
A_2 = (0,1] \setminus A_1.
\]
Now,
\begin{equation}\label{eq:cmuabove2}
\sup\limits_{A_2}f\left(\frac{1}{a(\varepsilon)}\right)+ C\eps^{N-4} \le
f\left(\frac{1}{2\mu^{\frac{N}{2}-1}}\right) + C = \frac{1-\sigma}{N}\mu^{1-\frac{N}{2}} + C,
\end{equation}
where $0<\sigma<1$ is a universal constant, explicit in $N$. On the other hand, using
again \eqref{est:struwe}, we have that, if $\mu^{**}$ (and thus $\eps\in A_1$) is sufficiently small,
then
\[
\eps \in A_1 \implies \eps \le C'\mu^{\frac{N-2}{4}},
\]
for some $C'>0$. This yields
\begin{equation}\label{eq:cmuabove3}
\sup\limits_{A_1}f\left(\frac{1}{a(\varepsilon)}\right)+ C\eps^{N-4} \le
f\left(\frac{1}{\mu^{\frac{N}{2}-1}}\right) + C'' \mu^{\frac{(N-2)(N-4)}{4}} =
 \frac{1}{N}\mu^{1-\frac{N}{2}} + C'' \mu^{\frac{(N-2)(N-4)}{4}}.
\end{equation}
Finally, plugging \eqref{eq:cmuabove2} and \eqref{eq:cmuabove3} into \eqref{eq:cmuabove1}, and taking
$\mu^{**}$ small enough, we have that the lemma follows, in case $N\ge5$, for every $0<\mu<\mu^{**}$.

\underline{Case 2: $N=4$}. By \eqref{est:struwe}, in this case we have
\begin{equation*}
    \begin{aligned}\displaystyle
 c_{\mu}\leq \sup\limits_{0<\varepsilon\le 1}E_{\mu}\left(\frac{u_{\varepsilon}}{\|u_{\varepsilon}\|_{2}}\right)\leq S^{2} \sup\limits_{0<\varepsilon\le 1}\left[\frac{1}{2}\frac{1}{a(\varepsilon)}-\frac{\mu}{4}\frac{1}{(a(\varepsilon))^{2}}+C\frac{1}{|\ln \varepsilon|}\right].
    \end{aligned}
\end{equation*}
Arguing as in the proof of Case 1, and taking into account that now $a(\eps)\le 2\mu^{\frac{N}{2}-1}$
implies $\eps^2 \ln \eps \le C' \mu^{\frac{N}{2}-1}$, we obtain the required estimate.

\underline{Case 3: $N=3$}.
As it is well known, in dimension $N=3$ the estimates in \eqref{est:struwe} need to be improved.
Precisely, with the same definitions as above, we argue similarly as in the proof of
\cite[eqs. (1.27)-(1.29)]{BrezisNirenberg}, although we need a small modification in that argument, 
to take into account Remark \ref{rmk:why_R}: we assume that $\overline{B}_R\subset B_{R_\Omega} 
\subset \Omega$ (where, up to a translation, the balls are centered at $0$), 
we write $\tau=R_\Omega-R>0$ and we choose
\[
\eta(x)=
\begin{cases}
1  &  x \in B_\tau\\
\cos\dfrac{\pi}{2R}\left( |x| - \tau\right)   &  x \in B_{R_\Omega}\setminus B_\tau
\end{cases}
\] 
in the definition of $u_\eps$, see equation \eqref{eq:def_u_eps}.
We obtain
\begin{equation*}
    \begin{aligned}\displaystyle
\|\nabla u_{\varepsilon}\|_{2}^{2}&=S^{3/2}+ c \varepsilon \int\limits_{\tau}\limits^{\tau+R}|\eta^{\prime}|^{2}dr +O(\varepsilon^{2}),\\
\|u_{\varepsilon}\|_{2^{*}}^{2^{*}}&=S^{3/2}+O(\varepsilon^{3}),\\
\|u_{\varepsilon}\|_{2}^{2}&=c\varepsilon \int\limits_{\tau}\limits^{\tau+R}\eta^{2}dr+O(\varepsilon^{2}),
    \end{aligned}
\end{equation*}
where $c= 3^{1/2}\cdot 4\pi$.

Now, writing again $a(\varepsilon)=\|u_{\varepsilon}\|_{2}^{2}$, we have
\begin{equation*}
    \begin{aligned}\displaystyle
    c_{\mu}\leq \sup\limits_{0<\varepsilon\le 1}E_{\mu}\left(\frac{u_{\varepsilon}}{\|u_{\varepsilon}\|_{2}}\right)&=\sup\limits_{0<\varepsilon\le 1}\left[\frac{1}{2}\frac{S^{3/2}+c \varepsilon \int\limits_{\tau}\limits^{\tau+R}|\eta^{\prime}|^{2}dr +O(\varepsilon^{2})}{a(\varepsilon)}-\frac{\mu}{6}\frac{S^{3/2}+O(\varepsilon^{3})}{a(\varepsilon)^{3}}\right]\\
    &=S^{3/2}\sup\limits_{0<\varepsilon\le1}\left[\left(\frac{1}{2}\frac{1}{a(\varepsilon)}-\frac{\mu}{6}\frac{1}{a(\varepsilon)^{3}}\right)
    +O(\varepsilon)+C\mu\right]+\frac{\int\limits_{\tau}\limits^{\tau+R}|\eta^{\prime}|^{2}dr}{\int\limits_{\tau}\limits^{\tau+R}\eta^{2}dr}.
    \end{aligned}
\end{equation*}
Recalling the definition of $\eta$ in this case, we have that
\[
\frac{\int\limits_{\tau}\limits^{\tau+R}|\eta^{\prime}|^{2}dr}{\int\limits_{\tau}\limits^{\tau+R}\eta^{2}dr}=
\frac{\pi^2}{4R^2}
=\frac{1}{4}\lambda_{1}(B_{R}),\]
which implies that
\begin{equation*}
    \begin{aligned}\displaystyle
 c_{\mu}\leq \sup\limits_{0<\varepsilon\le 1}E_{\mu}\left(\frac{u_{\varepsilon}}{\|u_{\varepsilon}\|_{2}}\right)\leq S^{3/2}\sup\limits_{0<\varepsilon\le 1}\left[\left(\frac{1}{2}\frac{1}{a(\varepsilon)}-\frac{\mu}{6}\frac{1}{a(\varepsilon)^{3}}\right)
    +C\varepsilon+C\mu\right]+\frac{1}{4}\lambda_{1}(B_{R}),
    \end{aligned}
\end{equation*}
where $B_{R}\subset \Omega$. Then, arguing as in Case 1, taking into account that $a(\eps)\le 2\mu^{\frac{N}{2}-1}$ implies $\eps \le C' \mu^{1/2}$, and choosing a possible smaller value of
$\mu^{**}$, we conclude the proof.
%\begin{equation*}
%    \begin{aligned}\displaystyle
%c_{\mu}&\leq S^{3/2}  \max\left[ \max\limits_{0<a(\varepsilon)\leq2\mu^{N/2-1}}g\left(\frac{1}{a(\varepsilon)}\right),\max\limits_{2\mu^{N/2-1}<a(\varepsilon)<a(\bar{\varepsilon})}g\left(\frac{1}{a(\varepsilon)}\right)\right]+\frac{1}{4}\lambda_{1}(B_{R})\\
%&\leq\frac{1}{3}S^{3/2}\mu^{-1/2}+C\mu^{1/2}+\frac{1}{4}\lambda_{1}(B_{R}).
%    \end{aligned}
%\end{equation*}
\end{proof}

\section{Proof of the main results}\label{sec:proof_main_res}
\subsection{Proof of Theorem \ref{Th:mini}}
In this section, we assume that $\Omega\subset\R^N$ is a bounded domain and prove that, if $\mu <\mu^*$,
then the local minimum level $m_{\mu}$, defined in \eqref{eq:minimumlevel}, is achieved. Then Theorem \ref{Th:mini} will follow by
the change of variables \eqref{eq:changeofvariables}.

By Lemma \ref{geo-struc}, there exists a
minimizing sequence
$(v_{n})_{n}\subset \Bcal_{\bar\alpha}$ associated to $m_{\mu}$. Since both $E_\mu$ and
$\Bcal_{\bar\alpha}$ are even, without loss of generality we can assume that $v_n\ge 0$ for
every $n$. Then, by Ekeland's principle, we can find another minimizing
sequence $(u_{n})_{n}\subset\Bcal_\alpha$ which is a Palais-Smale sequence for $E_\mu$ at level
$m_\mu$, and
\begin{equation}\label{eq:posmin}
\|u_n - v_n\|_{H^1_0}\to 0
\end{equation}
as $n\to+\infty$.

(Actually, Ekeland's principle holds for global minimizers; on the other hand, in view of Lemma
\ref{geo-struc}, it is not difficult to modify $E_\mu$ outside $\Bcal_{\bar\alpha}$ in such a way
that $m_\mu$ is the global infimum of the modified functional.)

Using once again Lemma \ref{geo-struc}, we also find that
\begin{equation}\label{bound:u:min}
\|\nabla u_{n}\|_{2}^{2}<\bar{\alpha}=S^{N/2}\mu^{1-\frac{N}{2}},
\end{equation}
which implies that $u_{n}$ is bounded in $H^{1}_{0}(\Omega)$. As a consequence,
also
\[
\lambda_n = \int_{\Omega}|\nabla u_n|^{2}dx-\mu\int_{\Omega}|u_n|^{2^{*}}dx
\qquad\text{is bounded, }
\]
so that we are in a position to apply Corollary \ref{cor:split}. Then, there exists
$u^{0}\in M$ such that, up to a subsequence, $u_{n}\rightharpoonup u^{0}$ in
$H^{1}_{0}(\Omega)$. Finally, assume that $u_n \not \to u^0$ strongly in $H^1_0$. Then
\eqref{eq:split_norm} yields
\[
0<\lambda_1(\Omega)\le\|\nabla u^0\|_2^2 \le \bar \alpha - S^{N/2}\mu^{1-\frac{N}{2}} =0,
\]
a contradiction. Then $u_n  \to u^0$ strongly in $H^1_0$, and the theorem follows.

\subsection{Proof of Theorem \ref{Th:mp2}}\label{sec:th2}
In this section, we assume that $\Omega\subset\R^N$ is a bounded domain, with $N\ge3$. Moreover,
in case $N=3$, we also assume that
\begin{equation}\label{eq:assumpt_ball_N=3}
\Omega\supset \overline{B}_R,\qquad\text{ with }\lambda_1(B_R) < \frac43 \lambda_1(\Omega).
\end{equation}
Notice that this can be done in view of assumption \eqref{eq:BRBR}, by taking $R_\Omega - R$ sufficiently small. 

We will show that there exists a threshold $\mu^{**}>0$ and a set $S_0$ such that
\begin{enumerate}
\item for every $0<\mu\le\mu^{**}$, $S_0\cap \left(0,\mu\right)$ has positive Lebesgue measure;
\item if either $\mu \in S_0$, or $\mu\in \overline{S}_0$ and it is the limit of an increasing sequence of
elements of $S_0$, then there exists a second solution of \eqref{eq:main1}, at the mountain pass
level $c_\mu$.
\item for any other $\mu\in \overline{S}_0$, there exists a second solution of \eqref{eq:main1},
having energy level $\ell\ge \frac{1}{N}S^{N/2}\mu^{1-\frac{N}{2}}$ (but maybe $\ell\neq c_\mu$).
\end{enumerate}
As a consequence, Theorem \ref{Th:mp2} will follow by the usual change of variables
\eqref{eq:changeofvariables}.

Let $\mu^{**}>0$ be such that Lemmas \ref{lem:w0} and \ref{lem:w1} hold true, and consider any interval $J= \left(\frac{\mu_0}{2},\mu_0\right)$, with $\mu_0\le\mu^{**}$. Finally, let $c_\mu$ be defined
as in \eqref{level:mp}, with $w_{0},w_{1}\geq0$ as in Lemmas \ref{lem:w0} and \ref{lem:w1}. With
these assumptions, we are in a position to apply Lemma \ref{JJ:theo}. As a consequence, we deduce that
for almost every $\mu\in J$, $c_\mu$ is differentiable and there exists a bounded Palais-Smale sequence
$(u_{n})_{n}$ for $E_{\mu}$ at level $c_{\mu}$, with
\begin{equation}\label{bound:u:mp}
\|\nabla u_{n}\|_{2}^{2} \le 2c_{\mu}-2c^{\prime}_{\mu}\mu
+6\mu +o_n(1).
\end{equation}
Actually, since, given $\mu$, the argument above is independent of $J\ni\mu$ (see Remark \ref{rmk:c_mu not depend on mu0}), we can assume that such a Palais-Smale sequence exists for almost every
$\mu\in (0,\mu^{**})$.

Since $(u_n)_n$ is bounded, the corresponding sequence of multipliers $(\lambda_n)_n$, defined
as in \eqref{eq:lambda_def}, is bounded too.
We infer the existence of $u^0\in M$ such that, up to subsequences,
\begin{equation}\label{eq:weak_thm2}
u_{n}\rightharpoonup u^{0} = u^{0}_\mu \ \ \text{weakly in } H^{1}_{0}(\Omega).
\end{equation}
Moreover $u^0$ is a positive solution of \eqref{eq:main1}, because of
Corollary \ref{cor:split}, property \ref{positivejnjn} in Lemma \ref{JJ:theo} and the strong maximum
principle. Notice that the whole discussion remains true in case we consider a smaller value, although
positive, of $\mu^{**}$. In the following, we will often need to do that, in order to exploit the
asymptotic behavior of $c_\mu$ as $\mu\to0$ to infer finer properties of the functions we consider.

Now, if we could prove that the convergence in \eqref{eq:weak_thm2} is strong, then $u^0$ would be a
solution of \eqref{eq:main1} at level $c_\mu$ (and thus different
from the local minimizer at level $m_\mu$ that we obtained in the previous section). As we noticed,
we can do this for every $\mu$ only assuming further assumptions on $\Omega$, as in the next section. 
In the generality we are considering here,
we will obtain strong convergence only for a subset of $(0,\mu^{**})$ of positive measure. The
key result in this direction is the following estimate, inspired by \cite[Lemma 2]{MR1245097}, which bounds the value of $c'_\mu$ to obtain
a sharper bound for $\|\nabla u_{n}\|_{2}^{2}$ in \eqref{bound:u:mp}.

\begin{lemma}\label{monoty}
Let $g(\mu)$, $h(\mu)$ be defined as in Lemma \ref{energy:mp}, and let us consider a possibly smaller
value of $\mu^{**}$ in such a way that $g$ is non-increasing on $(0,\mu^{**})$.
Then there exists a set $S_0\subset(0,\mu^{**})$  such that
\begin{enumerate}
\item $\left|S_0 \cap \left(0,\mu\right)\right|>0$ for every $\mu\in (0,\mu^{**})$ and
\item for every $\mu \in S_0$, $c_\mu$ is differentiable and
\begin{equation*}
-c'_\mu = |c^{\prime}_\mu|\leq (1+\delta(\mu))|g^{\prime}(\mu)|=-(1+\delta(\mu))g^{\prime}(\mu),
\end{equation*}
where
\begin{equation}\label{eq:delta_mu}
\delta(\mu):=
\begin{cases}
\mu^p & N\ge5,\ p=\dfrac{N}{2}-\dfrac{1}{2},\\
\mu|\ln \mu|^{-1/2} & N=4,\smallskip\\
\delta_0\mu^{1/2} & N=3,
\end{cases}
\end{equation}
and, when $N=3$, we can choose as $\delta_0$ any number such that
\begin{equation}\label{eq:delta0_N=3}
\frac13 S^{3/2} \delta_0 > \frac14 \lambda_1(B_R),
\end{equation}
up to taking $\mu^{**}$ suitably small, depending on $\delta_0$ (here $B_R$ satisfies \eqref{eq:assumpt_ball_N=3}).
\end{enumerate}
\end{lemma}
\begin{proof}
By contradiction, we assume that there exists $\mu_{0}\le\mu^{**}$ such that, for almost every
$\mu\in(0,\mu_{0})$,
\begin{equation*}
-c^{\prime}_{\mu}\ge-(1+\delta(\mu))g^{\prime}(\mu).
\end{equation*}
Without loss of generality, we can assume that $\delta$ is non-decreasing on $(0,\mu_0)$.
Let now
\begin{equation}\label{eq:def_gamma}
\gamma=
\begin{cases}
2 & N\ge 4,\\
>1,\text{ to be chosen below,} & N=3,
\end{cases}
\end{equation}
and take $\mu_1>0$ in such a way that
$(\mu_1,\gamma\mu_1)\subset (0,\mu_0)$. Since $c_\mu$ is non-increasing, 
by Rademacher's theorem we infer, for a.e. $\mu_1$,
\begin{equation}\label{eq:S_0_1}
    \begin{aligned}\displaystyle
c_{\mu_{1}}&\geq\int\limits_{\mu_{1}}\limits^{\gamma\mu_{1}}-c^{\prime}_{\mu}dx+c_{\gamma\mu_{1}}\ge\int\limits_{\mu_{1}}\limits^{\gamma\mu_{1}}-(1+\delta(\mu))g^{\prime}(\mu)dx+c_{\gamma\mu_{1}}\\
&\geq(1+\delta(\mu_1))\left[g(\mu_{1})-g(\gamma\mu_{1})\right]+c_{\gamma\mu_{1}}\\
&=g(\mu_{1})+\delta(\mu_1)[ g(\mu_{1})-g(\gamma\mu_{1})] -[g(\gamma\mu_{1})-c_{\gamma\mu_{1}}].
    \end{aligned}
\end{equation}
Recalling Lemma \ref{energy:mp} we have that, for every $\mu$,
\[
g(\mu)-h(\mu) \le c_\mu \le g(\mu);
\]
on the other hand, by direct calculations,
\[
g(\mu_{1})-g(\gamma\mu_{1}) = \frac{1}{N}S^{N/2}(1-\gamma^{1-\frac{N}{2}})\mu_{1}^{1-\frac{N}{2}}
+ h(\mu_{1})-h(\gamma\mu_{1}) \ge K_\gamma\mu_{1}^{1-\frac{N}{2}} - h(\gamma\mu_1),
\]
where
\begin{equation}\label{eq:Kgamma}
K_\gamma = \frac{1}{N}S^{N/2}(1-\gamma^{1-\frac{N}{2}}).
\end{equation}
In particular, when $N=3$, we can use condition \eqref{eq:delta0_N=3} in order to
choose $\gamma$ in such a way that
\begin{equation}\label{eq:N=3_Kgamma}
K_\gamma \delta_0 > \frac14 \lambda_1(B_R).
\end{equation}

Collecting the above inequalities, we have that for every $\eps>0$ there exists $\mu_0$ sufficiently small such that, for a.e. every $\gamma\mu_1\le\mu_0$ \eqref{eq:S_0_1} rewrites as
\[
0 \ge K_\gamma\delta(\mu_1)\mu_{1}^{1-\frac{N}{2}} - (1+\delta(\mu_1))h(\gamma\mu_1)
\ge K_\gamma\delta(\mu_1)\mu_{1}^{1-\frac{N}{2}} - (1+\eps) h(\gamma\mu_1).
\]
Since $K_\gamma>0$, we easily get a contradiction for $\mu_0$ sufficiently small,
recalling that
\[
 \delta(\mu)\mu^{1-\frac{N}{2}} =
 \begin{cases}
 \mu^{1/2} & N\ge5,\\
|\ln \mu|^{-1/2} & N=4,\\
\delta_0 & N=3,
 \end{cases}
\qquad
h(\gamma\mu) =
\begin{cases}
C'\mu^{\frac{(N-2)(N-4)}{4}}, \ \ \ &\mbox{if} \ N\geq5,
\smallskip\\
C'|\ln \mu|^{-1}, \ \ \ &\mbox{if} \ N=4,\\
\frac{1}{4}\lambda_{1}(B_{R})+C'\mu^{1/2}&\mbox{if} \ N=3:
\end{cases}
\]
this is clear for $N=4$, and also for $N\ge5$, since $\frac{1}{2}< \frac{3}{4}\le \frac{(N-2)(N-4)}{4}$ in such case; on the other hand, in dimension $N=3$,
this follows from \eqref{eq:N=3_Kgamma}.
\end{proof}

Based on the previous lemma, we can prove strong convergence in $S_0$.
\begin{lemma}\label{lem:ex_in_S0}
For every $\mu \in S_0$, the convergence in \eqref{eq:weak_thm2} is strong. Moreover,
$u^0= u^0_\mu$ satisfies
\begin{equation}\label{eq:norm_in_S0}
\|\nabla u_\mu^0\|_2^2 \le S^{N/2}\mu^{1-\frac{N}{2}} + \lambda_1(\Omega) - \eps,
\end{equation}
for  $\eps>0$ suitably small (and for a possibly smaller value of $\mu^{**}$).
\end{lemma}
\begin{proof}
Let $\mu\in S_0$. Using Corollary \ref{cor:split}, we assume by contradiction
that $u_n \not \to u^0$ strongly in $H^1_0$.

We first deal with the case $N\ge5$. By Lemma \ref{monoty}, we have (recall that $p=\frac{N}{2}-
\frac{1}{2}$)
\begin{equation*}
    \begin{aligned}\displaystyle
-c^{\prime}_{\mu}\mu\leq-(1+\mu^p)g^{\prime}(\mu)\mu\leq (1+\mu^p)\left[\frac{N-2}{2N}S^{N/2}\mu^{1-\frac{N}{2}}+\frac{C(N-2)(N-4)}{4}\mu^{\frac{(N-2)(N-4)}{4}}\right].
    \end{aligned}
\end{equation*}
Using \eqref{bound:u:mp}, \eqref{eq:mpfromabove} and taking into account that $(N-2)(N-4)\ge3$ for $N\ge5$ we obtain, as $n \to +\infty$
\begin{equation}\label{eq:weakboundS_0}
    \begin{aligned}\displaystyle
\|\nabla u_{n}\|_{2}^{2} &\le 2c_{\mu}-2c^{\prime}_{\mu}\mu+6\mu+o_{n}(1)\\
&\le S^{N/2}\mu^{1-\frac{N}{2}}+ \frac{N-2}{N}S^{N/2}\mu^{1/2}+6\mu+(1+\mu^p)C'\mu^{\frac{(N-2)(N-4)}{4}}+o_{n}(1)
\\
&\le S^{N/2}\mu^{1-\frac{N}{2}}+ C''\mu^{1/2}+o_{n}(1)
\\
&\le S^{N/2}\mu^{1-\frac{N}{2}}+ \lambda_1(\Omega) - \eps+o_{n}(1),
    \end{aligned}
\end{equation}
for $\eps>0$ suitably small and for every $\mu\in S_0\cap (0,\mu^{**})$, as long as
$\mu^{**}$ is sufficiently small. Taking the weak limit of \eqref{eq:weakboundS_0} we
obtain \eqref{eq:norm_in_S0}. On the other hand, combining \eqref{eq:weakboundS_0} and
\eqref{eq:split_norm}, we have that $\mu\in S_0\cap (0,\mu^{**})$
yields
\[
\lambda_1(\Omega)\le\|\nabla u^0\|_2^2 \le \liminf \|\nabla u_{n}\|_{2}^{2} - S^{N/2}
\mu^{1-\frac{N}{2}}\le\lambda_1(\Omega)-\eps,
\]
a contradiction. This implies strong convergence, and concludes the proof in case $N\ge5$.

A similar argument works also when $N=4$. Indeed, in such case, by Lemma \ref{monoty} we have
\begin{equation*}
    \begin{aligned}\displaystyle
-c^{\prime}_{\mu}\mu\leq-\left(1+\frac{\mu}{|\ln \mu|^{1/2}}\right)g^{\prime}(\mu)\mu\leq \left(1+\frac{\mu}{|\ln \mu|^{1/2}}\right)\left[\frac{1}{4}S^{2}\mu^{-1}+\frac{C}{|\ln \mu|^{2}}\right],
    \end{aligned}
\end{equation*}
which implies, by \eqref{bound:u:mp},
\begin{equation*}
    \begin{aligned}\displaystyle
\|\nabla u_{n}\|_{2}^{2} &\le 2c_{\mu}-2c^{\prime}_{\mu}\mu+6\mu+o_{n}(1)\\
&\le S^{2}\mu^{-1}+ 6\mu+\frac{C}{|\ln \mu|^{1/2}}+\frac{C}{|\ln \mu|}+\frac{C}{|\ln \mu|^2}+\frac{C\mu}{|\ln \mu|^{5/2}}+o_{n}(1)
\\
&\le S^{2}\mu^{-1}+ C'|\ln \mu|^{-1/2}+o_{n}(1)\\
&\le S^{2}\mu^{-1}+ \lambda_1(\Omega) - \eps+o_{n}(1)
    \end{aligned}
\end{equation*}
as $n \to +\infty$, for a suitable $\eps>0$, $\mu\in S_0\cap (0,\mu^{**})$ and $\mu^{**}>0$ possibly smaller. Then we can conclude as in the case $N\ge 5$.

Finally, in case $N=3$, analogous calculations yield
\begin{equation*}
    \begin{aligned}\displaystyle
-c^{\prime}_{\mu}\mu\leq-\left(1+\delta_0\mu^{1/2}\right)g^{\prime}(\mu)\mu\leq \left(1+
\delta_0\mu^{1/2}\right)\left[\frac{1}{6}S^{3/2}\mu^{-1/2}+C\mu^{1/2}\right],
    \end{aligned}
\end{equation*}
whence
\begin{equation*}
\|\nabla u_{n}\|_{2}^{2} \le S^{3/2}\mu^{-1/2}+ \frac{1}{2} \lambda_1(B_R) + \frac13 S^{3/2} \delta_0  + 6\mu+C\mu^{1/2}+o_{n}(1)
\end{equation*}
as $n \to +\infty$. To conclude, we have by \eqref{eq:assumpt_ball_N=3} that 
\[
\frac12\lambda_1(B_R) < \frac23  \lambda_1(\Omega),
\]
and moreover we can take $\mu^{**}$, $\delta_0$ in Lemma \ref{monoty}
in such a way that (see \eqref{eq:delta0_N=3})
\[
\frac34 \lambda_1(B_R) < S^{3/2} \delta_0 \le  \lambda_1(\Omega) -6\eps,
\]
for $\eps>0$ sufficiently small. We infer
\begin{equation*}
    \begin{aligned}\displaystyle
\|\nabla u_{n}\|_{2}^{2} &\le S^{3/2}\mu^{-1/2}+
\frac{2}{3} \lambda_1(\Omega) +
\frac{1}{3} \lambda_1(\Omega) -2\eps + C\mu^{1/2} + o_{n}(1)\\
&\le S^{3/2}\mu^{-1/2}+\lambda_1(\Omega) -\eps  + o_{n}(1),
    \end{aligned}
\end{equation*}
for $\mu$ small enough, and the lemma follows.
\end{proof}

To conclude the proof of Theorem \ref{Th:mp2}, we have to show that a second solution exist
also for $\mu \in \overline{S}_0 \setminus S_0$, where $S_0$ is as in the previous
lemmas, and to check whether its energy level is
$c_{\mu}$.

Take $\hat{\mu}\in \overline{S}_0\setminus S_0$; by definition, there exists a sequence $(\mu_{n})_{n}\subset S_0$ such that $\mu_{n}\rightarrow\hat{\mu}$.
Now, by Lemma \ref{lem:ex_in_S0}, we know that for every $n$ there exists $u_{\mu_{n}}^{0}$, which satisfies
\[
E_{\mu_{n}}(u^{0}_{\mu_{n}}) = c_{\mu_{n}}, \qquad \nabla_M E_{\mu_{n}}(u^{0}_{\mu_{n}}) =
\nabla_{H^1_0} I_{\lambda_n,\mu_n}(u^{0}_{\mu_{n}})=0;
\]
hence, arguing as in the proof of Corollary \ref{cor:split}, we infer that $(u_{\mu_{n}}^{0})_{n}$ is a Palais-Smale sequence for $E_{\hat\mu}$ on $M$, at level
\[
\ell:=\lim_{n} c_{\mu_n}\ge
\frac{1}{N}S^{N/2}\hat\mu^{1-\frac{N}{2}}> m_{\hat\mu}.
\]
Notice that, since $c_\mu$ is continuous from the left,
we can assure that $\ell = c_{\hat\mu}$ only if $\mu_n < \hat \mu$ for every large $n$.

In any case, by \eqref{eq:norm_in_S0}, $(u_{\mu_{n}}^{0})_{n}$ is bounded in M. From this, there exists $u_{\hat{\mu}}^{0} \in M$ such that
\[
u_{\mu_{n}}^{0} \rightharpoonup u_{\hat{\mu}}^{0} \ \ \ \text{in } M.
\]
Then, Corollary \ref{cor:split} applies to $(u_{\mu_{n}}^{0})_{n}$, and
combining \eqref{eq:norm_in_S0} and \eqref{eq:split_norm} we have that the above convergence
is strong, thus concluding the proof of the theorem.

\subsection{Proof of Theorem \ref{Th:mp}}

In this section, we keep using the same assumptions as in Section \ref{sec:th2}, and we further
assume that $\Omega$ is star-shaped. We will show that, if $\mu^{**}$ is sufficiently small, then \eqref{eq:main1} has a second solution, at the mountain pass level $c_\mu$, for every
$\mu \in (0,\mu^{**})$.

Notice that the first part of the proof of Theorem \ref{Th:mp2}, up to \eqref{bound:u:mp}, holds true
also in this case. In particular, we already know that for almost every $\mu\in (0,\mu^{**})$ there
exists a bounded Palais-Smale sequence $(u_n)_n$ for $E_\mu$ on $M$, at level $c_\mu$, having weak
limit $u^0=u^0_\mu$. For further reference, we denote with $F$ such set:
\begin{equation}\label{eq:def_set_F}
F:=\{\mu>0:c_\mu\text{ admits a PS-sequence satisfying \eqref{bound:u:mp}, \eqref{eq:weak_thm2}}\},
\quad \text{with } |(0,\mu^{**})\setminus F|=0.
\end{equation}

As before, we will first show, for every $\mu\in F$, strong convergence in $H^1_0(\Omega)$,
thus providing existence of a mountain pass solution; next, we will extend the existence result
to every $\hat\mu \in (0,\mu^{**})\setminus F$, by approximating $\hat\mu$ with a sequence in $F$.

We exploit the further assumption about $\Omega$ in the next lemma.

\begin{lemma}\label{lem:conseq_Poho}
If $\Omega$ is star-shaped, $\mu>0$ and $u\in H^1_0(\Omega)$, $u>0$ is a solution of
\eqref{eq:main1} then:
\begin{enumerate}
\item $\lambda>0$;
\item $E_\mu(u) \ge \dfrac{1}{N} \lambda_1(\Omega)$;
\item $\|\nabla u\|_2^2 \le N E_\mu(u)$.
\end{enumerate}
\end{lemma}
\begin{proof}
By the well known Pohozaev identity, see e.g. \cite[Ch. III, Lemma 1.4]{Struwebook},
we have that the solution $u$ satisfies
\[
2\lambda\int_{\Omega}u^{2}dx= \int_{\partial\Omega}\left|\frac{\partial u}{\partial \nu}\right|^{2} x\cdot\nu \,dx
\]
(here we are assuming, without loss of generality, that $\Omega$ is star-shaped with respect
to $0\in\Omega$).
Then, by $x\cdot\nu\ge0$ for all $x\in\partial \Omega$, we have $\lambda>0$.

Testing the equation \eqref{eq:main1} with $u$ we obtain
\[
\mu\int_{\Omega}|u|^{2^{*}}dx = \int_{\Omega}|\nabla u|^{2}dx - \lambda,
\]
and finally
\[
E_\mu(u) = \frac12 \int_{\Omega}|\nabla u|^{2}dx-\frac{\mu}{2^{*}}\int_{\Omega}|u|^{2^{*}}dx =
\frac1N \int_{\Omega}|\nabla u|^{2}dx+ \frac{1}{2^{*}}\lambda \ge \frac1N \int_{\Omega}|\nabla u|^{2}dx
\ge \frac1N \lambda_1(\Omega). \qedhere
\]
\end{proof}

\begin{lemma}\label{lem:strongconv_final}
Let $\mu\in F \cap (0,\mu^{**})$. Then, for a
possible smaller value of $\mu^{**}$, $u_n \to u^0$ strongly in $H^1_0(\Omega)$.
\end{lemma}

\begin{proof}
Assume the convergence is not strong. Using Corollary \ref{cor:split}, we have
\begin{equation}\label{eq:forcontr}
c_\mu=\lim E_{\mu}(u_{n})\geq E_{\mu}(u^{0})+ \frac{1}{N} S^{N/2}\mu^{1-\frac{N}{2}}
\ge \frac{1}{N}\lambda_{1}(\Omega) + \frac{1}{N} S^{N/2}\mu^{1-\frac{N}{2}},
\end{equation}
where we used Lemma \ref{lem:conseq_Poho}.
We will get a contradiction using Lemma \ref{energy:mp}. We distinguish two cases.

When $N\geq4$, by Lemma \ref{energy:mp}, we can take $\mu^{**}$ small enough such that, for any $0<\mu<\mu^{**}$, we have
\begin{equation*}
c_{\mu}\leq\frac{1}{N}S^{N/2}\mu^{1-\frac{N}{2}} + \frac{1}{N+1}\lambda_{1}(\Omega),
\end{equation*}
in contradiction with \eqref{eq:forcontr}.

On the other hand, when $N=3$, Lemma \ref{energy:mp} yields
\begin{equation*}
c_{\mu}\leq\frac{1}{3}S^{3/2}\mu^{-1/2}+\frac{1}{4}\lambda_{1}(B_{R})+C\mu^{1/2},
\end{equation*}
which is again in contradiction with \eqref{eq:forcontr}, for $\mu^{**}$ sufficiently small,
recalling assumption \eqref{eq:assumpt_ball_N=3}.
\end{proof}

To conclude the proof of the theorem, we argue as in the end of the proof of Theorem \ref{Th:mp2} to
show that, for every $\mu\in(0,\mu^{**})$, there exists a solution at level $c_\mu$.

Let $\hat{\mu}\in (0,\mu^{**})\setminus F$. Since $F$ has full measure, there exists an increasing
sequence $(\mu_{n})_{n}\subset F$ such that
\[
\frac{\hat \mu}{2} \le \mu_{n}\nearrow\hat{\mu}.
\]
Using Lemma
\ref{lem:strongconv_final}, we have that for every $n$ there exists $u_{\mu_{n}}^{0}$ such that
 \[
E_{\mu_{n}}(u^{0}_{\mu_{n}}) = c_{\mu_{n}}, \qquad \nabla_M E_{\mu_{n}}(u^{0}_{\mu_{n}}) =
\nabla_{H^1_0} I_{\lambda_n,\mu_n}(u^{0}_{\mu_{n}})=0;
\]
once again, we deduce that $(u_{\mu_{n}}^{0})_{n}$ is a Palais-Smale sequence for
$E_{\hat\mu}$ on $M$, at level $c_{\hat\mu}$ (recall that $c_\mu$ is continuous from the left,
due to Lemma \ref{JJ:theo}).

Finally, we can apply Lemma \ref{lem:strongconv_final} to each $u^0_{\mu_n}$ to obtain
\[
\|\nabla u_{\mu_n}^0\|_2^2 \le N E_\mu(u_{\mu_n}^0) = N c_{\mu_n} \le N c_{\hat\mu/2}.
\]
Then $(u_{\mu_{n}}^{0})_{n}$ is uniformly bounded. From this, there exists $u_{\hat{\mu}}^{0} \in M$ such that
\[
u_{\mu_{n}}^{0} \rightharpoonup u_{\hat{\mu}}^{0} \ \ \ \text{in } M.
\]
Then, using Corollary \ref{cor:split}, arguing as in the proof of Lemma \ref{lem:strongconv_final}, we have that the above convergence is strong, thus concluding the proof of the theorem.
\bigskip

\textbf{Acknowledgments.} Work partially supported by the MUR-PRIN-20227HX33Z 
``Pattern formation in nonlinear phenomena'', the Portuguese government through FCT/Portugal
under the project PTDC/MAT-PUR/1788/2020, the MUR grant Dipartimento di Eccellenza 2023-2027, 
the INdAM-GNAMPA group.\bigskip

\textbf{Disclosure statement.} The authors report there are no competing interests to declare.

\bibliography{normalized}{}
\bibliographystyle{abbrv}
\medskip
\small

\begin{flushright}
{\tt dario.pierotti@polimi.it}\\
{\tt gianmaria.verzini@polimi.it}\\
{\tt junwei.yu@polimi.it}\\
Dipartimento di Matematica, Politecnico di Milano\\
piazza Leonardo da Vinci 32, 20133 Milano, Italy.
\end{flushright}

\end{document}